\documentclass[12pt]{article}
\usepackage{amsmath}
\usepackage{amsfonts}
\usepackage{amssymb}
\usepackage{amsthm}
\usepackage{amstext}
\usepackage{color}
\usepackage{graphicx}

\usepackage{color}

\newtheorem{teorema}{Theorem}[section]
\newtheorem{lema}[teorema]{Lemma}
\newtheorem{corolario}[teorema]{Corollary}
\newtheorem{proposicion}[teorema]{Proposition}
\theoremstyle{definition}
\newtheorem{remark}[teorema]{Remark}
\newtheorem{ejemplo}[teorema]{Example}

\def\div{\mbox{div}\,}
\def\grad{\mbox{grad}\,}

\def\C{\mathbb{C}}
\def\D{\mathbb{D}}
\def\T{\mathbb{T}}

\def\R{\mathbb{R}}

\begin{document}

\title{Bounded Extremal Problems in Bergman and Bergman-Vekua spaces}
\author{Briceyda B. Delgado\thanks{Regional mathematical center of Southern Federal University, Bolshaya Sadovaya, 105/42, Rostov-on-Don, 344006, Russia, briceydadelgado@gmail.mx} \and Juliette Leblond\thanks{INRIA Sophia Antipolis M\'editerran\'ee, Team FACTAS, BP 93, 06902 Sophia Antipolis, France,  juliette.leblond@inria.fr} }
\date{}
\maketitle
\begin{abstract}
 We analyze Bergman spaces $A^p_f(\D)$ of generalized analytic functions of solutions to the Vekua equation $\overline{\partial}w=(\overline{\partial}f/f)\overline{w}$ in the unit disc of the complex plane, for Lipschitz-smooth non-vanishing real valued functions $f$ and $1 < p < \infty$.
  We consider a family of bounded extremal problems (best constrained approximation) in the Bergman space $A^p(\D)$ and in its generalized version $A^p_f(\D)$, that consists in approximating a function in  subsets of $\D$ by the restriction of a function belonging to $A^p(\D)$ or $A^p_f(\D)$ subject to a norm constraint. Preliminary constructive results are provided for $p=2$.
\end{abstract}
\mbox{}functions; Bergman spaces; Vekua equation; Bergman-Vekua spaces; bounded extremal \\problems.\\ 
\mbox{}\\
\noindent \textbf{Classification:} 30G20, 30H20, 30E10, 41A29. 


\section{Introduction}

Solutions to the general Vekua equation $\overline{\partial}w=\alpha {\overline{w}+\beta w}$, whose theory was introduced in \cite{Bers1953,Vekua1959} in the complex plane, are the so-called pseudo-holomorphic (or generalized holomorphic, or analytic) functions, see also \cite[Ch. 6]{Colton1980}. In this work, we study the Bergman spaces $A_f^p(\D)$ 
of generalized holomorphic functions 
solutions to $\overline{\partial}w=\alpha_f\overline{w}$ in the unit disc $\D$, with $\alpha_f=\overline{\partial}f/f=\overline{\partial}\log f$, where $f$ is a non-vanishing Lipschitz-smooth real valued function, whence $\alpha_f$ is uniformly bounded in $\D$, 
and establish some of their basic properties.  Whenever $w=w_0+i w_1$ satisfies  $\overline{\partial}w=\alpha_f  
\,\overline{w}$, the real and imaginary parts of $w_0/f+i fw_1$ are solutions to the elliptic conductivity equation $\nabla \cdot \left( \sigma \nabla u \right) = 0$ with conductivities $\sigma = f^2$ and $1/f^2$, respectively, see e.g. \cite{BLRR2010, Krav2009}. This generalizes the links between holomorphic and harmonic functions.

In \cite{BLRR2010}, Hardy spaces of generalized holomorphic functions were studied on $\D$ and on Dini-smooth simply connected domains 
(see also references therein and \cite{BFL2011} for an extension to finitely connected domains). Generalized Smirnov-Hardy spaces of Lipschitz simply connected domains were further considered in \cite{BBL2016}.  Boundary value problems were considered as well in \cite{BBL2016, BFL2011, BLRR2010} either of Dirichlet or Neumann type, or from partial overdetermined Cauchy type data, which is handled through best constrained approximation problems (bounded extremal problems, {\sl BEP}) in these Hardy classes.
In \cite[Th.\ 4.2]{Hen2002}, necessary and sufficient conditions to guarantee the uniqueness of the solution to Dirichlet problem in Bergman and Sobolev spaces of Lipschitz domains were provided.

The present work is an extension to Bergman spaces $A^p(\D)$ and generalized Bergman spaces $A_f^p(\D)$, $1<p<\infty$, of results obtained for Hardy spaces $H^p(\D)$ in \cite{ABL1992, BLP1996, BL1998}, and for Hardy spaces of  generalized holomorphic functions in \cite{FLPS2011}. More precisely, we formulate and solve best constrained approximation problems in these Bergman spaces. They could furnish an appropriate setting to formulate inverse problems for some elliptic partial differential equations (PDE) in $\D$, or conformally equivalent smooth enough domains, of conductivity or Shr\"odinger type, when not only partial boundary Cauchy data are available but also some measurements within the domain.

The bounded extremal problems {\textsl{$f$-BEP}} under study are as follows. Being given a partition $\D$ in two  (non-empty) domains $K\subseteq \D$ and $J=\D\setminus \overline{K}$, let $h_K\in L^p(K)$ and $h_J\in L^p(J)$, $1 < p < \infty$. We aim at finding a function belonging to  $A_f^p(\D)$ whose distance to $h_{J}$ in $L^p(J)$ does not exceed some fixed positive real-valued quantity and which is as close as possible to $h_{K}$ in  $L^p(K)$ under this constraint. Whenever $f\equiv 1$, it holds that $A_f^p(\D)=A_0^p(\D)=A^p(\D)$, and we will refer to the corresponding problem as {\textsl{BEP}}. 

  This work is structured as follows. In Section \ref{sec:preliminaries}, we recall a few classical results concerning Bergman and Hardy spaces, mostly on the disc, but on other domains as well. Section \ref{sec:Bergman-Vekua spaces} contains an analysis of the Bergman-Vekua spaces $A_f^p(\D)$ of pseudo-holomorphic functions.
In Section \ref{sec:BEP-Bergman}, we consider bounded extremal problems {\textsl{BEP}} in Bergman spaces $A^p(\D)$ and related constructive issues for the case $p=2$, for which an explicit formula is obtained in terms of the Bergman projection. 
We also provide a preliminary study of the corresponding problem {\textsl{$f$-BEP}} in generalized Bergman-Vekua spaces $A_f^p(\D)$.
Finally, a brief conclusion is given in Section \ref{sec:conclu}.

\section{Preliminaries}\label{sec:preliminaries}
We first introduce some notation, then recall some results concerning Bergman and Hardy spaces, as well as some density properties in these spaces of analytic functions. 

\subsection{Notation}

Let $C^{0,\gamma}(\overline{\D})$ be the space of H\"older-continuous functions on the closure $\overline{\D}$ of $\D$, with exponent $0<\gamma<1$. 
For $1 < p\leq \infty$, the classical Sobolev space $W^{1,p}(\D)$ denotes the space of functions in $L^p(\D)$ such that their derivatives (in the distribution sense) also belong to $L^p(\D)$, see \cite{Fournier1978,Brezis2011} ($W^{1,\infty}(\D)$ is the space of Lipschitz-smooth functions on $\D$).
Let $W_0^{1,p}(\D)$ be the closure of the set of $C^\infty$-smooth functions with compact support in $\D$, taken in  $W^{1,p}(\D)$.

A domain $\Omega$ is a Lipschitz domain if its boundary $\partial \Omega$ can be viewed as the graph of Lipschitz-smooth function up to a rotation and if the domain is on one side of its boundary only, see \cite[Sec.\ 1.1]{Grisvard1985}, \cite{Pommerenke1992}.

{Let $\chi_\Omega$ be} the characteristic function of $\Omega \subset \D$. For $h_1$ defined on $\Omega$ and $h_2$ defined on $\D \setminus {\overline{\Omega}}$, let $h_1\vee h_2$ be the function defined on $\D$ which coincides with $h_1$ on $\Omega$ and with $h_2$ on $\D \setminus {\overline{\Omega}}$.

On $\C \simeq \R^ 2$, define the derivative $\overline{\partial}$ with respect to $\bar{z}$, by $\overline{\partial} = 1/2(\partial_x + i \partial_y)$, where $\partial_x$ and $\partial_y$ are the partial derivatives w.r.t. the variables $x$ and $y$, respectively, and similarly for the derivative w.r.t. ${z}$, ${\partial} = 1/2(\partial_x - i \partial_y)$.
We will designate the operators $\div$ and $\grad$ by $\nabla \cdot$ and $\nabla$, respectively, for $\nabla = (\partial_x, \partial_y)$.

\subsection{Bergman spaces}\label{subsec:Bergman spaces}
The theory of Bergman spaces in this preliminary subsection is based on \cite[Chap.\ 1, Sec.\ 1.1, 1.2]{zhu2000}. Other references are 
\cite{Alpay2015,DuSch2004}.
Let $dA(z)=\frac{1}{\pi}dx dy=\frac{1}{\pi}r dr d\theta$ be the normalized Lebesgue measure on the unit disc $\D \subseteq \C$, with $z=x+iy=re^{i\theta}$. For $1\leq p<\infty$, the \textit{Bergman space} $A^p(\D)$ is the subspace of analytic functions in $L^p(\D)=L^p(\D,dA)$ with the usual norm given by:
\begin{align*}
   \|g\|_{A^p(\D)}^p=\|g\|_{L^p(\D)}^p=\int_{\D}{|g(z)|^p \, dA(z)} =\frac{1}{\pi}\int_0^1{\left(\int_0^{2\pi}{|g(re^{i\theta})|^p \, d\theta}\right) r\, dr}.
\end{align*}
By \cite[Prop.\ 1.2]{zhu2000}, the Bergman space $A^p(\D)$ is closed in $L^p(\D)$ (for $0<p<\infty$).
The following integral operator $P$ is the orthogonal projection from $L^2(\D)$ onto $A^2(\D)$ \cite[Prop.\ 1.4]{zhu2000}:
\begin{align}\label{eq:Bergman_projection}
   P\,g(z)=\int_{\D}{\frac{g(\zeta)}{(1-z\bar{\zeta})^2}\, dA(\zeta)},
\end{align}
for all $z\in \D$ and for all $g\in L^2(\D)$. Moreover, $P$ is a bounded projection operator from $L^2(\D)$ to $A^2(\D)$. The operator $P$ is the known \textit{Bergman projection} over $\D$ and for every $z, \zeta \in \D$, the function $K(z,\zeta)=1/(1- \bar{z} \zeta)^2$ is called the Bergman reproducing kernel on $\D$. 
The following orthogonal decomposition holds (see \cite[Lemma\ 2.1]{Hen2002}):
\[
   L^2(\D)=A^2(\D)\oplus \partial W_0^{1,2}(\D),
\]
Although the Bergman projection $P$ is originally defined on $L^2(\D)$, 
the integral formula can be extended to $L^1(\D)$, using that $A^2(\D)$ is dense in $A^1(\D)$. Furthermore, for $1<p<\infty$, $P$ is a bounded projection operator from $L^p(\D)$ onto $A^p(\D)$, see \cite[Th.\ 1.10]{zhu2000}.

We may also analyze Bergman spaces $A^p(\Omega)$ for any bounded domain (open connected set) $\Omega\subseteq \C$, that consist in the space of all analytic functions in $\Omega$ such that:
\begin{align*}
   \|g\|_{A^p(\Omega)}^p=\|g\|_{L^p(\Omega)}^p= \int_{\Omega}{|g(z)|^p \, dA(z)} < \infty,
\end{align*}
see \cite{Bergman1950, Hen2002}. Notice that $A^p(\Omega)$ is non-empty if $\Omega$ has at least one boundary component that consists of more than a single point.
Moreover, $A^p(\Omega)$ is closed in $L^p(\Omega)$. 

More is needed for $A^p(\Omega)$ to be non-trivial, like $\C \setminus \Omega$  thick enough, see \cite{Hen2002} and references therein. 

In the development of this work, we will use the well-known following uniqueness principle for analytic functions: 
\begin{lema}(\cite[Th.\ 10.18]{Rudin1987}, \cite[Th.\ 9.15]{GuHaSpr2008}). \label{prop:Identity principle}
Let $\Omega$ be a domain of $\C$. Let $g$ be an analytic function in $\Omega$, and 
$   Z(g)=\left\{z\in \Omega \colon g(z)=0\right\}$ be 
the set of zeros of $g$. Then either $Z(g)=\Omega$ or $Z(g)$ has no limit point in $\Omega$. 
\end{lema}

\subsection{Hardy spaces}\label{subsec:Hardy spaces}

Now we will see some aspects of the classical theory of Hardy spaces on the unit disc. For references about Hardy spaces consult \cite[Ch. 20]{Conway1995}, \cite{Duren1970, Garnett1981, Koosis1980}. 
For $1 \leq p<\infty$,  the \textit{Hardy space} $H^p(\D)$ is the space of  analytic functions $g \colon \D\to \C$ such that:
\begin{align*}
  \|g\|_{H^p(\D)}=\sup_{0<r<1} \left[\frac{1}{2\pi}\int_0^{2\pi}
	{\left|g\left(re^{i\theta}\right)\right|^{p} \, d\theta} \right]^{1/p}  <\infty.
\end{align*}
Moreover, every analytic function $g\in H^p(\D)$ has a radial limit \cite[p.\ 70]{Koosis1980}, that is $g_b(e^{i\theta}):=\lim_{r\rightarrow 1^-} g(re^{i\theta})$ exists for almost all $e^{i\theta}$ in $\T = \partial \D$.
The Hardy space $H^p(\D)$ is a closed subspace of $L^p(\T)$, $p\geq 1$, \cite[Th.\ 1.5]{Conway1995}. Further, since $\|g\|_{A^p(\D)} \leq \|g\|_{H^p(\D)}$, $H^p(\D) \subset A^p(\D)$ continuously. Analogously to \eqref{eq:Bergman_projection}, there exists an orthogonal projection $P_{H^2} \colon L^2(\T)\to H^2(\D)$ given by:
\begin{align*}
  P_{H^2}g(z)=\frac{1}{2\pi}\int_0^{2\pi} {\frac{g(e^{i\theta})}{1-ze^{-i\theta}} \, d \theta}=\frac{1}{2i\pi}\int_\T {\frac{g(\zeta)}{\zeta-z}} \, d \zeta \, .
\end{align*} 
In particular, the reproducing kernel on $H^2(\D)$ is given by $K_{H^2}(z, \zeta) = 1/(1-\bar{z} \, \zeta)$.
\begin{remark}\label{rmk:radial-limit-Bergman}
  One important difference between the Hardy and Bergman spaces is that functions in the Bergman space $A^p(\D)$ do not necessarily have radial limit or non-tangential limit almost everywhere on $\T$. In fact, there exists a function in $A^2(\D)$ which fails to have radial limits at every point of $\T$ \cite{Mac1962}.
\end{remark}
%
The projection operator $P_{H^2}$ can also be extended from $L^p(\T)$ to $H^p(\D)$ whenever $1 < p < \infty$ (as for the projection $P$ from $L^p(\D)$ to $A^p(\D)$). 
In Hardy spaces as in Bergman spaces, the uniqueness property holds true from (non-empty) open subsets of $\D$ (see Lemma \ref{prop:Identity principle}). However, for Hardy spaces $H^p(\D)$, it remains valid 
on subsets of $\T$ of positive measure.

\subsection{Bergman and Hardy spaces, density results}\label{subsec:density-results}

The following significant density result of polynomials holds in Bergman classes:
\begin{proposicion}\label{prop:densitypolynew}
  Let $\Omega \subset \C$ be a bounded simply connected Lipschitz domain. Then the set $\mathcal{P}$ of polynomials is dense in $A^p(\Omega)$, for $1 \leq p < \infty$.
\end{proposicion}
\begin{proof}
In particular, $\partial \Omega$ is a rectifiable Jordan curve whence $\Omega$ is a Carath\'e\-odory domain, see \cite{BBL2016, Grisvard1985, Pommerenke1992}. The result then follows from \cite[Sec.\ 5]{Qiu}. 
\end{proof}
As a consequence of Mergelyan's theorem \cite[Th.\ 20.5]{Rudin1987} or Nehari's criteria, we also have:
\begin{lema} (\cite[Ch.\ 10]{Alpay2015}). \label{lem:Mergelyan}
  Let $\Omega \subset \C$ be a bounded domain (resp. compact subset) such that $\C \setminus \Omega$ is connected (resp. the closure of a domain). Then the set $\mathcal{P}$ of polynomials is dense in $A^2(\Omega)$.
\end{lema}
\begin{remark}\label{rmk:densitypolynew}
  Actually, that $\Omega$ is a Carath\'e\-odory domain is a sufficient condition for the conclusion of Proposition \ref{prop:densitypolynew} to hold true.

  However, it turns out that under the present assumptions, 
  $\Omega$ is also a Smirnov domain (bounded by a Jordan curve and such that both the conformal map $\Omega \to \D$ and its inverse are outer $H^p$ functions), see \cite[Lem. 5.1]{BBL2016}, \cite{Duren1970, Grisvard1985}. 
  We then deduce from \cite[Th.\ 10.6]{Duren1970} that $\mathcal{P}$ is also dense in Hardy-Smirnov spaces. This is the framework in which inverse boundary value problems and their extensions to interior data can be set. 


  Finally, in the Hilbert case $p=2$ and if $\C \setminus \Omega$ is 
  connected, the density result holds true from Lemma \ref{lem:Mergelyan} (using the reproducing kernel).  Notice that 
   more can be said for $p$ close to 2, see \cite{Hen2002}.
\end{remark}
Examples of subsets of $\D$ that satisfy the assumptions of the above Proposition \ref{prop:densitypolynew} and Lemma \ref{lem:Mergelyan} are balls $\Omega=a \, \D$ with $0<a<1$ (simply connected radial subsets), and ``angular subsets'' $\Omega=\left\{z\in \D \colon -\theta<\arg(z)< \theta\right\}$, with $\theta\in (0,\pi)$.\\

Let us denote by $H^p(\D)|_{\Omega}$ and $A^p(\D)|_{\Omega}$, the spaces of restrictions on $\Omega\subseteq \D$ of $H^p(\D)$ and $A^p(\D)$ functions, respectively. Notice that $A^p(\D)|_{\Omega}\subset A^p(\Omega)$ is a proper subset if $\Omega \subset \D$ is a strict subset because not all the functions in $A^p(\Omega)$ admit an analytic continuation to $\D$ \cite[Ch.\ 9, Sec.\ 3]{Conway1978}. However:
\begin{proposicion}\label{prop:density}
  Let $1\leq p<\infty$ and let $\Omega\subseteq \D$ be a (non-empty) domain. Then,
\begin{enumerate}
   \item [$(a)$] $H^p(\D)$ is dense in $A^p(\D)$ in $L^p(\D)$ norm.
\end{enumerate}
If besides the assumptions of Proposition \ref{prop:densitypolynew}, Lemma \ref{lem:Mergelyan}, or Remark \ref{rmk:densitypolynew} 
are satisfied, then: 
\begin{enumerate}
   \item [$(b)$] $A^p(\D)|_{\Omega}$ is dense in $A^p(\Omega)$ in $L^p
	 (\Omega)$ norm. 
\end{enumerate}
\end{proposicion}
\begin{proof} 
  $(a)$ Since $\mathcal{P} \subset H^p(\D)\subset A^p(\D)$
    and $\mathcal{P}$ is dense in $A^p(\D)$ (see \cite[Prop. 2.6]{zhu2004}), we have:
\[
   A^p(\D)=\overline{\mathcal{P}}\subseteq \overline{H^p(\D)}
	\subseteq \overline{A^p(\D)}.
\]
Thus, because $A^p(\D)$ is a closed subset of $L^p(\D)$, we get that $\overline{H^p(\D)}=A^p(\D)$ in $L^p(\D)$.

$(b)$ By Proposition \ref{prop:densitypolynew} or Lemma \ref{lem:Mergelyan}, taking the closure in $L^p(\Omega)$, we have:
\[
   A^p(\Omega)=\overline{\mathcal{P}|_\Omega}\subseteq \overline{A^p(\D)|_\Omega}
	 \subseteq \overline{A^p(\Omega)} \text{ in } L^p(\Omega) \, .
\]
Finally, because $A^p(\Omega)$ is a closed subset of $L^p(\Omega)$, we obtain $\overline{A^p(\D)|_{\Omega}}=A^p(\Omega)$ in 
$L^p(\Omega)$.
\end{proof}

\section{Generalized Bergman-Vekua spaces}\label{sec:Bergman-Vekua spaces}

Let now $1< p<\infty$ and let $f \in W^{1,\infty}(\D)$ be a Lipschitz-smooth real valued function on $\D$ 
with $0 <1/k\leq |f|\leq k<1$ in $\D$. The \textit{generalized Bergman-Vekua spaces $A_f^p(\D)$ consists} in those Lebesgue measurable functions $w=w_0+iw_1 \in L^p(\D)$
such that:
\begin{align}\label{eq:Vekua_2}
   \overline{\partial} w=\alpha_f \, \overline{w}, \quad \text{ with }   \alpha_f=\frac{\overline{\partial} f}{f} \in L^\infty(\D) \, ,
\end{align}
in the sense of distributions on $\D$.

Notice that \eqref{eq:Vekua_2} is equivalent to the system of generalized Cauchy-Riemann equations:
\begin{align*} 
   \nonumber \partial_x(fw_1)=-f^2\partial_y \left(\frac{w_0}{f}\right),\quad
	\partial_y(fw_1)=f^2\partial_x\left(\frac{w_0}{f}\right).
\end{align*}
It is worth mentioning that given $\alpha=\alpha_0+i\alpha_1$ such that $\partial_x \alpha_1-\partial_y\alpha_0=0$, then it is possible to construct a real valued function $f$ in terms of $\alpha$ satisfying the $\bar{\partial}$ problem: $\bar{\partial} (\log f) = \alpha$.

Moreover, $A_f^p(\D)$ is a non-trivial subspace over $\R$ because of $f, i/f\in A_f^p(\D)$.
The equation \eqref{eq:Vekua_2} was called the \textit{main Vekua equation} in \cite{Krav2009}, which plays an important role in the theory of pseudo-analytic functions \cite{Bers1953} (generalized analytic functions \cite{Vekua1959}). 



We define some integral operators as in \cite{BLRR2010, GuHaSpr2008}, 
which will be useful in the development of this work. Remind that in the complex case, for $w\in L^p(\D)$ (with $1<p<\infty$), the 
\textit{Teodorescu transform} is defined as follows, for $ z\in \D$:
\begin{align}\label{eq:Teodorescu_operator}
   T_{{\D}}[w](z)=\int_{\D}{\frac{w(\zeta)}{z- \zeta} \, dA(\zeta)}, \end{align}
(using that $d\zeta \land d\bar{\zeta}=-2\pi i dA(\zeta)$). The boundedness of $T_{\D} \colon L^p(\D) \to W^{1,p}(\D)$ as well as the compactness of $T_{\D}$ in $L^p(\D)$ were established in \cite[Prop.\ 5.2.1]{BLRR2010}. Furthermore, the operator $T_{\D}$ is a right inverse of the $\overline{\partial}$ operator, which means $\overline{\partial}T_{\D}[w]=w$, for all $w\in L^p(\D)$ and $1<p<\infty$. 
\begin{lema}\label{lemma:Pf}
Let $\alpha_f \in L^{\infty}(\D)$ and $w\in L^p(\D)$. Then:
\begin{align*}
w\in A_f^p(\D) \Leftrightarrow w-T_{\D}[{\alpha_f}\overline{w}] \in A^p(\D). 
\end{align*}
\end{lema}
\begin{proof}
Let $w\in L^p(\D)$. From the fact that $\overline{\partial}T_{\D}=I$, if $w$ is a weak solution of \eqref{eq:Vekua_2} on $\D$, then $\overline{\partial}(w-T_{\D}[{\alpha_f}\overline{w}])=0$ on $\D$, so by the Weyl's Lemma \cite[Th.\ 24.9]{Forster1981} we have that $w-T_{\D}[{\alpha_f}\overline{w}]$ is analytic on $\D$. Moreover, since $T_{\D}\colon L^p(\D)\to L^p(\D)$ and $\alpha_f \overline{w}\in L^p(\D)$, we have  $w-T_{\D}[{\alpha_f}\overline{w}]\in L^p(\D)$. Conversely, if $w-T_{\D}[{\alpha_f}\overline{w}]\in A^p(\D)$, then $w$ is a solution to \eqref{eq:Vekua_2}. 
\end{proof}

\subsection{Properties of $A_f^p(\D)$}\label{subcsec:properties}
In the literature, the term ``similarity principle''  is used to refer to the factorization of pseudo-analytic functions through analytic functions. 
See for example \cite{Bers1953, Colton1980, Vekua1959}, \cite[Sec.\ 4.3]{Krav2009}, for the similarity principle of the general Vekua equation: 
\begin{align}\label{eq:Vekua-general}  
	\overline{\partial}w={\alpha \overline{w}+ \beta w}, 
\end{align}
with coefficients $\alpha $ and $\beta$ satisfying a H\"{o}lder condition and vanishing outside a large disk.

Notice that our Bergman-Vekua spaces $A_f^p(\D)$ are introduced in the particular case of  equation \eqref{eq:Vekua-general}, where $\beta=0$ and $\alpha=\alpha_f$, see equation \eqref{eq:Vekua_2}. 
In this case, another version of the similarity principle was given in \cite[Th.\ 4.2.1]{BLRR2010} for Hardy spaces of generalized analytic functions. Here, we have:
\begin{teorema}(Similarity principle)
\cite{Bers1953, Colton1980, Krav2009, Vekua1959}\label{th:Similarity_principle} 
Let $1< p<\infty$ and let $w\in A_f^p(\D)$. Then $w$ can be represented as
\begin{align}\label{eq:Similarity_principle}
   w(z)=e^{s(z)}\,F(z), \quad z\in \D,
\end{align}
where $s\in W^{1,l}(\D)$ for all $l\in (1,\infty)$ and $F$ is holomorphic in $\D$. Moreover, $F\in A^p(\D)$. 

Reciprocally, let $F\in A^p(\D)$ and $f$ be such that $\alpha_f \in C^{0,\rho}(\overline{\D})$ for $0< \rho < 1$. Then there exists a function $s \in C^{0,\gamma}(\overline{\D})$, with $0<\gamma<1$, such that $w=e^{s} F\in A_f^p(\D)$.
\end{teorema}
We omit the proof of this theorem. Indeed, the proof of the first part is analogous to the one given in the previous references,  see \cite[Th.\ 6.4]{Colton1980}. The second part, a partial converse of the first part, was proved in \cite[Th.\ 6.5]{Colton1980}.
Since $s \in L^\infty(\D)$, it holds that $w \in A_f^p(\D)$ if, and only if, $F \in A^p(\D)$.

A sufficient condition in order to ensure $\alpha_f \in C^{0,\rho}(\overline{\D})$ is that $f \in C^{1,\rho}(\overline{\D})$, together with the boundedness assumptions on $|f|$.

By the Sobolev imbedding Theorem \cite[Cor.\ 9.14]{Brezis2011}, 
if $l>2$ then $W^{1,l}(\D)\hookrightarrow C^{0,\gamma}(\overline{\D})$, with $0<\gamma= 1 - 2/l<1$.

Notice that the function $s$ in the factorization \eqref{eq:Similarity_principle} is a solution of the $\overline{\partial}$-problem: 
\begin{align} \label{eq:s_condition}
   \overline\partial s(z)={\alpha_f}(z) \, \frac{\overline{w(z)}}{w(z)}, \quad \text{ in }\D,
\end{align}
whence $s$ is unique up to a holomorphic additive function. In \cite[Th.\ 4.2.1]{BLRR2010} the function $s$ was given in terms of the Teodorescu operator $T_{\D}$ plus a strategic holomorphic function.
Furthermore,
\begin{align*}
  \|s\|_{L^{\infty}(\D)}\leq 4 \left\|\alpha_f\right\|_{L^{\infty}(\D)} \,.
\end{align*}
The following result is a consequence of {Lemma \ref{lemma:Pf}}.
\begin{proposicion}\label{prop:Bergman-generalized-closed}
Let $1< p<\infty$ and $f$ be such that $\alpha_f \in L^{\infty}(\D)$. Then $A_f^p({\D})$ is closed in $L^p(\D)$.
\end{proposicion}
\begin{proof}
  Let $\left\{w_n\right\}\subseteq A_f^p(\D)$ be a sequence such that $w_n\rightarrow w$ in $L^p(\D)$. By Lemma \ref{lemma:Pf}, $w_n-T_{\D}[\alpha_f\overline{w_n}]\in A^p(\D)$. Further, it converges to $w-T_{\D}[\alpha_f\overline{w}]$ in $L^p(\D)$. Indeed, 
    we have:
\[ 
\|w_n-T_{\D}[\alpha_f\overline{w_n}]-w+T_{\D}[\alpha_f\overline{w}]\|_{L^p(\D)} \leq \|w_n-w\|_{L^p(\D)} + \|T_{\D}[\alpha_f(\overline{w_n}-\overline{w})]\|_{L^p(\D)} \, ,
\]
which goes to 0 as $w_n\rightarrow w$ in $L^p(\D)$, by the continuity of the Teodorescu transform $T_{\D}$.
Consequently, using the closedness of the classical Bergman spaces $A^p(\D)$ in $L^p(\D)$, we get that $w-T_{\D}[\alpha_f \overline{w}]\in A^p(\D)$. Again by Lemma \ref{lemma:Pf}, $w\in A_f^p(\D)$.
\end{proof}

From Proposition \ref{prop:Bergman-generalized-closed} and because closed subsets of reflexive (separable) spaces are reflexive (separable) \cite[Part\ 1, Ch.\ III, Sec.\ 1, Cor.\ 3]{Beauzamy1985} ( \cite[Part\ 1, Ch.\ III, Sec.\ 2, Prop.\ 2]{Beauzamy1985}), we have:
\begin{corolario}\label{cor:reflexive-separable}
Let $\alpha_f\in L^{\infty}(\D)$ and {$1<p<\infty$}. Then $A_f^p(\D)$ is a reflexive and {separable} Banach space.
\end{corolario}


\begin{remark}\label{rmk:radial-limit-generalized}
Following Theorem \ref{th:Similarity_principle} and Remark \ref{rmk:radial-limit-Bergman}, functions in the Bergman-Vekua space $A_f^2(\D)$ do not necessarily admit a radial or non-tan\-gen\-tial limit almost everywhere on $\T$.
This represents an important difference between the spaces $A_f^p(\D)$ and the generalized Hardy spaces defined in \cite{BLRR2010}, where the non-tangential limit always exists \cite[Prop.\ 4.3.1]{BLRR2010}. However, many other properties are preserved, for example, 
the zeroes of $w\in A_f^p(\D)$, $w \not \equiv 0$, are isolated in $\D$ \cite[Cor.\ 59]{Krav2009}, {which generalizes the uniqueness principle (Lemma \ref{prop:Identity principle})}.  
\end{remark}
{Furthermore, $G=w_0/f+ifw_1$ is a solution of the conjugate Beltrami equation:
\begin{align}\label{eq:Beltrami}
    \overline{\partial} G=\frac{1-f^2}{1+f^2}\overline{\partial G},
\end{align}
for every $w=w_0+iw_1\in A_f^p(\D)$, which establishes a correspondence between the solutions of \eqref{eq:Vekua_2} and \eqref{eq:Beltrami}. 
If $G\in W^{1,p}(\D)$ is a solution of  \eqref{eq:Beltrami}, then $G$ belongs to the Hardy space of solutions of the conjugate Beltrami equation analyzed in \cite{BLRR2010}. Therefore, $G$ satisfies the maximum principle, that is, $|G|$ cannot assume a relative maximum in $\D$ unless it is constant \cite[Prop.\ 4.3.1]{BLRR2010}.}

We know that given an arbitrary real valued harmonic function in a simply connected domain, we can construct a conjugate harmonic function explicitly such that the resultant pair of harmonic functions represent the real and imaginary parts of a complex holomorphic function. Since the real and imaginary parts of a solution to the main Vekua equation \eqref{eq:Vekua_2} are solutions of stationary Schr\"{o}dinger equations (see Lemma \ref{lemma:conductivity-Sch}), they are called \textit{conjugate metaharmonic functions}, see \cite{Dautray1985, Krav2009}.
\begin{lema}\label{lemma:conductivity-Sch}\cite[Th.\ 33]{Krav2009}\\
Let $w=w_0+iw_1 \in A_f^p(\D)$. Then $w_0$ and $w_1$ are weak solutions in $\D$ of the following equations:
\begin{align*}
\begin{matrix}
  \displaystyle  \nabla \cdot \left(f^2\,\nabla\left(\frac{w_0}{f}\right)\right) &=& \displaystyle 0, \quad
  \nabla \cdot \left(f^{-2}\,\nabla(fw_1)\right) &=& 0,
\end{matrix}
\end{align*}
or, equivalently:
\begin{align*}
  \begin{matrix}
    \displaystyle   \frac{\Delta \, w_0}{w_0}-\frac{\Delta \, f}{f} &=& \displaystyle 0, \quad
  \frac{\Delta\, w_1}{w_1}-\Delta \, \left(\frac{1}{f}\right) \, f &=& 0.
\end{matrix}
\end{align*}
\end{lema}
Notice that whenever $w$ 
is a solution to \eqref{eq:Vekua_2} with real and imaginary parts no identically vanishing, then $w$ can not be harmonic in $\D$ except if $f$ is constant. It is a direct consequence of Lemma \ref{lemma:conductivity-Sch} and the identity $\Delta(f^{-1})=-f^{-2}\,\Delta f+2\,f^{-3}\, |\nabla f|^2$.

\subsection{Dual space of {$A_f^2(\D)$}}\label{subsec:Dual}

By Proposition \ref{prop:Bergman-generalized-closed}, $A_f^2(\D)\subseteq L^2(\D)$ is a real Hilbert space and there exist orthoprojections:
\begin{align*}
   P_f &\colon L^2(\D) \to A_f^2(\D),\\
   Q_f &\colon L^2(\D) \to A_f^2(\D)^{\bot};
\end{align*}
such that $P_f+Q_f=I$ in $L^2(\D)$. We call $P_f$ to the \textit{Bergman-Vekua projection}. 
Expressions of the projection $P_f$ 
can be obtained, that may help to characterize $A_f^2(\D)$ and $A_f^2(\D)^{\bot}$.
\begin{proposicion}\label{prop:expP}
The restriction of $P_f$ to $A_f^2(\D)$ can be written as:  
\begin{equation}\label{eq:reduction-projection}
   P_f \, w = P \, w + Q \, \left[T_{\D}[\alpha_f \overline{w}]\right],
\end{equation}
with $P$ the Bergman projection given by \eqref{eq:Bergman_projection} and $Q=I-P$ its corresponding orthoprojection.
\end{proposicion}

\begin{proof}
By Lemma \ref{lemma:Pf}, we get that $w\in A_f^2(\D)$ if, and only if, 
  $w-T_{\D}[\alpha_f \overline{w}]\in A^2(\D)$, which implies that:
\begin{equation*}
    w-T_{\D}[\alpha_f \overline{w}] = P \, \left[ w-T_{\D}[\alpha_f 
		\overline{w}]\right].\qedhere
\end{equation*}
\end{proof}
From \eqref{eq:Similarity_principle}, for $w \in A_f^2(\D)$, we obtain the following implicit reproducing relation: 
\begin{align*}
w(z) =  e^{s(z)}\,F(z) &= e^{s(z)}\, \langle F{(\cdot)}, K(z, \cdot)\rangle_{L^2(\D)} = e^{s(z)}\, \langle w (\cdot)\, e^{-s(\cdot)}, K(z, \cdot)\rangle_{L^2(\D)}\\
&=\langle w(\cdot) , e^{\overline{s(z)-s(\cdot)}}\, K(z, \cdot)\rangle_{L^2(\D)}.
\end{align*}
However, observe that $s$ depends on $w$.

Unfortunately, we do not know how to characterize the orthogonal complement subset $A_f^2(\D)^{\bot}$ explicitly. Some orthogonal decompositions of generalized Vekua spaces are given in \cite{Sprossig1993}, in the context of Clifford algebras (when $\alpha=0$,  $\beta=\overline{\partial} p/p$ in \eqref{eq:Vekua-general}, and $p$ be a positive $C^{\infty}$ scalar function).

The following invariance formula for the Bergman-Vekua space $A_f^2(\D)$ 
is an immediate consequence of the orthogonality of the projections $P_f$
and $Q_f$:
\begin{align}\label{eq:invariance_formula_f}
   \text{Re }\int_{\D}{g(z)\overline{h(z)} \, dA(z)}=\text{Re }\int_{\D}{g(z)\overline{P_f \, h(z)} \,dA(z)}.
\end{align}
for every $h\in L^2(\D)$ and $g\in A_f^2(\D)$. When $f\equiv 1$, \eqref{eq:invariance_formula_f} reduces to the real part of the known invariance formula for Bergman spaces.

It is well-known that the dual space $(A^p(\D))^*$ is isomorphic to $A^q(\D)$ \cite[Th.\ 1.16]{zhu2000}, for $1< p<\infty$ and $q$ be the conjugate exponent of $p$, i.e. $1/p+1/q=1$. For the generalized Bergman-Vekua spaces, this result still holds for $p=2$:
\begin{proposicion}\label{prop:dual-space}
Let $\alpha_f\in L^{\infty}(\D)$. The dual of $A_f^2(\D)$ coincides with itself.
\end{proposicion}
\begin{proof}
Let $\Psi$ be a real valued bounded linear functional on $A_f^2(\D)$. By the Hahn-Banach extension theorem \cite[Th.\ 5.16]{Rudin1987}, $\Psi$ can be 
extended to $L^2(\D)$. Using that $L^2(\D)$ is the dual space of $L^2(\D)$, there exists a function $h\in L^2(\D)$ such that:
\begin{align*}
   \Psi(g)=\text{Re }\int_{\D}{g(z)\overline{h(z)} \, dA(z)}, \quad \forall g\in 
	 A_f^2(\D) \, .
\end{align*}
By the invariance formula \eqref{eq:invariance_formula_f}, we have:
\begin{align*}
   \Psi(g)=\text{Re }\int_{\D}{g(z)\overline{P_f \, h(z)} \, dA(z)}, \quad \forall g\in A_f^2(\D),
\end{align*}
and the result follows from the fact that $P_f\, h \in A_f^2(\D)$.
\end{proof}
 From the fact that $A_f^2(\D)$ is separable (see Corollary \ref{cor:reflexive-separable}), it is straightforward that $A_f^2(\D)$ has an orthonormal basis \cite[Th.\ 5.11]{Brezis2011}, namely $\{e_n\}\subset A_f^2(\D)$. That is, the space spanned by $\{e_n\}$ is dense in $A_f^2(\D)$.
\begin{ejemplo}
\label{ex:sep}
Assume that $f$ admits the separable form $f(z)=\varrho(x)/\tau(y)$, $z = x+ i y$, where $\varrho$ and $\tau$ are  $C^2$ non-vanishing functions on $[-1,1]$. By \cite[Th. 33]{Campos2012}, any solution of the main Vekua equation \eqref{eq:Vekua_2} in $\D$ can be approximated arbitrarily closely on any compact subset of $\D$ by a finite real linear combination of the usual formal powers. 
Related examples together with approximation issues in Hardy spaces are considered in \cite{BFL2011, FLPS2011} for $f(x,y)=\varrho(x)$.
\end{ejemplo}

\section{Bounded extremal problems in Bergman spaces}\label{sec:BEP-Bergman}

Let $1< p<\infty$. Let $K \subset \D$ and $J=\D \setminus \overline{K}$ be two (non-empty) domains.
Let $h_J\in L^p(J)$, $0<M<\infty$, and let us introduce the family of  functions:
\begin{align}\label{eq:family_BEP}
   \mathcal{F}_{M,h_J}^p:=\left\{g\in A^p(\D) \colon \|h_J-g\|_{L^p(J)}\leq M \right\}. 
\end{align}
Let $h_K\in L^p(K)$. The bounded extremal problem {\textsl{BEP}} in Bergman spaces consists in finding $g_0\in \mathcal{F}_{M,h_J}^p$ such that:
\begin{align}\label{eq:BEP}
   \|h_K-g_0\|_{L^p(K)}=\text{min} \, \left\{\|h_K-g\|_{L^p(K)} \colon g\in \mathcal{F}_{M,h_J}^p\right\} \, . 
\end{align}
Similar {\textsl{BEP}} have been studied for example in \cite{ABL1992, BL1998, BLP1996} in the Hardy spaces $H^p(\D)$ with $1 \leq p \leq \infty$, 
where $K$ and $J$ are complementary subsets of the unit circle $\T$ with  positive measure.
\subsection{Existence and uniqueness}\label{subsec:existence and uniqueness}

Observe first that for the approximation set $\mathcal{F}_{M,h_J}^p$ to be non-empty, it has to be assumed that $M$ is larger than the distance of $h_J$ to the closure of $A^p(\D)|_{J}$ in $L^p(J)$. If $h_J \in A^p(\D)|_{J}$, whence in particular if $h_J = 0$, this is of course granted for all $M >0$ (actually for all $M \geq 0$). We also see from the density result of Proposition \ref{prop:density}, $(b)$, that if $J$ is a bounded simply connected Lipschitz domain or if $p=2$ and $\C \setminus J$  is a connected set, this can also be ensured for all $M >0$ by assuming that $h_J \in A^p(J)$ (see also Proposition \ref{prop:densitypolynew}, Lemma \ref{lem:Mergelyan}, Remark \ref{rmk:densitypolynew}). In such situations, the projection operator $L^p(J) \to A^p(J)$ is bounded, and if $h_J \in L^p(J)$, we can replace it by its projection $h_J^{+}$  onto $A^p(J)$ provided that $M$ is large enough ($M > \|h_J - h_J^{+}\|_{L^p(J)}$).\\

\textsl{In the remaining of this section, we assume that $M>0$, $\infty>p>1$, and $h_J\in L^p(J)$ are such that $\mathcal{F}_{M,h_J}^p \neq \emptyset$.}\\

We need the following result:
\begin{lema}\label{lemma-closed-1}
The approximation set $\mathcal{F}_{M,h_J}^p|_K$ is closed in $L^p(K)$.  
\end{lema}
\begin{proof} The convex set $\mathcal{F}_{M,h_J}^p$ is a closed subset of $A^p(\D)$, as follows from the closedness of $A^p(\D)$ in $L^p(\D)$ and the weak-closedness of balls in $L^p$ for $1<p<\infty$, \cite[Th.\ 3.17]{Brezis2011}. It is therefore weakly-closed, see \cite[Th.\ 3.7]{Brezis2011}.
  Let then $f_{K,n} \in \mathcal{F}_{M,h_J}^p|_K$ such that $f_{K,n} \to f_K$ in $L^p(K)$. 
  Therefore, there exists $\{f_n\}\subseteq \mathcal{F}_{M,h_J}^p$ such that $f_{K,n}=f_n|_K$ and this sequence is bounded in $L^ p(\D)$, since it converges on $K$ and satisfies the norm constraint on $J$. The weak-compactness of balls in $L^p(\D)$ \cite[Th.\ 3.17]{Brezis2011} yet ensures that $\{f_n\}$ admits a weakly-convergent subsequence, say to some limit $f \in L^p(\D)$. Since $\mathcal{F}_{M,h_J}^p$ is weakly-closed, we must have   $f \in \mathcal{F}_{M,h_J}^p$; furthermore, $f_{K}=f|_K$ then belongs to $ \mathcal{F}_{M,h_J}^p|_K$.
  %
We could also argue using  Montel's theorem \cite{Conway1978}.
\end{proof}
  

\begin{teorema}\label{th:existence-uniqueness-BEP}
  There exists a unique solution $g_0\in \mathcal{F}_{M,h_J}^p$ to the {\textsl{BEP}} \eqref{eq:BEP}.
Moreover, if $h_K \not \in \mathcal{F}_{M,h_J}^p|_K$, then the constraint is saturated: 
\begin{align*}
   \|g_0-h_J\|_{L^p(J)}=M.
\end{align*} 
\end{teorema} 
\begin{proof}
Since $\mathcal{F}_{M,h_J}^p|_K$ is a closed and convex subset of the uniformly convex space  $L^p(K)$, there exists a best approximation projection on $\mathcal{F}_{M,h_J}^p|_K$, see \cite[Part\ 3, Ch.\ II, Sec.\ 1, Prop.\ 5]{Beauzamy1985}. 
This ensures the existence of a minimizing element $g_0\in \mathcal{F}_{M,h_J}^p$, while uniqueness is then granted by uniform convexity and  Lemma \ref{prop:Identity principle}. 
\end{proof}
Notice that the last conclusion still holds if $h_K \in \mathcal{F}_{M,h_J}^p|_K$ and but is such that the function $h \in \mathcal{F}_{M,h_J}^p$ such that $h|_K = h_K$ satisfies $\|h-h_J\|_{L^p(J)}=M$; in this case of course $g_0 = h$.

The {\textsl{BEP}} \eqref{eq:BEP} is a particular extremal problem of the family studied in \cite{Chalender2004, Chalender2011} in Banach spaces, of which constructive aspects are developed in \cite[Sec. 7.3]{Chalender2011}.
We now consider its solution when $p=2$.
	
\subsection{Constructive aspects for $p=2$}\label{subsec:Hilbert case}

Let $p=2$. Let $M >0$, $h_J \in A^2(J)$ and $h_K \in A^2(K) \setminus \mathcal{F}_{M,h_J}^2|_K$.
\begin{teorema}\label{th:Theorem-solution-BEP}
The solution $g_0$ to the {\textsl{BEP}} \eqref{eq:BEP} is given by:
\begin{align}\label{eq:solution_p_2}
   g_0=(I+\lambda \,P \,\chi_J)^{-1} \, P(h_K\vee (\lambda+1)\,h_J)\, ,
\end{align}
   for the unique $\lambda\in (-1,\infty)$ such that $\|g_0-h_J\|_{L^2(J)}=M$.
\end{teorema}
\begin{proof}
  For $p=2$, the {\textsl{BEP}} \eqref{eq:BEP} has a solution given by \cite[Th.\ 7.3.2]{Chalender2011}, since $P \chi_K\cdot$ and $P \chi_J\cdot$ are the adjoint operators of the restriction operators from $A^2(\D)$ to $A^2(K)$ and $A^2(J)$, respectively.
  This is  to the effect that:
  \begin{align} \label{eq:approximation_p_2}
    (\lambda +1) \, P(\chi_J g_0 -0\vee h_J)=-P(\chi_K g_0 -h_K \vee 0)\, , \end{align}
  holds true for the unique $\lambda> -1$ such that $\|g_0-h_J\|_{L^2(J)}=M$.
  This can also be written as:
  \[
  (I+\lambda \,P \,\chi_J) \, g_0=P(h_K\vee (\lambda+1)\,h_J)\, . \] 
    The Toeplitz operator ${\cal T}_\phi \colon A^2(\D) \to A^2(\D)$, with symbol $\phi \in L^{\infty}(\D)$ is defined as follows, \cite{Lark1986, Vasilevsky2008}:
\begin{align*}  
   {\cal T}_\phi(g)=P(\phi \, g) \, ,
\end{align*}
where $P$ is the Bergman projection, defined  in \eqref{eq:Bergman_projection}. Since the Bergman projection $P$ has norm $1$ in $L^2(\D)$ \cite[Ch.\ 7]{zhu2007}, we have:
\[ 
   \|{\cal T}_\phi \|\leq \|\phi \|_{L^{\infty}(\D)}.
\] 
Take $\phi=\chi_J$. 
Since $\|P(\chi_J \cdot)\|\leq 1$ and the Toeplitz operator ${\cal T}_{\chi_J} \cdot = P(\chi_J \cdot)$ is self-adjoint on $A^2(\D)$, its spectrum is contained in 
$[0,1]$, see \cite{Lark1986}.
Therefore, $I+\lambda \, {\cal T}_{\chi_J}$ is invertible in $A^2(\D)$ for  every $\lambda>-1$. Thus the solution $g_0$ to the {\textsl{BEP}} \eqref{eq:BEP} satisfies \eqref{eq:solution_p_2}. 
\end{proof}
Furthermore, \eqref{eq:approximation_p_2}
implies that if $M\rightarrow 0$, then $\lambda \rightarrow \infty$, while if $M\rightarrow \infty$, then $\lambda\rightarrow -1$.

Observe that formula \eqref{eq:solution_p_2} is similar to the expression obtained in the context of Hardy spaces in \cite[Th.\ 4]{BL1998}. 
\begin{ejemplo}\label{ex:radial}
Suppose that $J$ is a radial subset of $\D$, that is $J=a \, \D$ where $0<a<1$. The radial symbol $\chi_J$ is defined as from \cite[Cor.\ 6.1.2]{Vasilevsky2008}, the Toeplitz operator ${\cal T}_{\chi_J}$ is compact in $A^2(\D)$, an its discrete spectrum is given by
 $  \left\{a^{n+1} \colon n\in 	\mathbb{Z}_+ \right\}\cup \left\{0\right\} \subset [0,1]$.
\end{ejemplo}

\subsection{Bounded extremal problem in Bergman-Vekua spaces}\label{sec:f-BEP}

Let $h_J\in L^p(J)$, $0<M<\infty$, and introduce the family of  functions in the Bergman-Vekua space 
$A_f^p(\D)$:
\begin{align}\label{eq:family_f-BEP}
   \mathcal{F}_{M,h_J}^{f,p}:=\left\{w\in A_f^p({\D}) \colon \|h_J-w\|_{L^p(J)}\leq M \right\}\, ,
\end{align}
that generalizes $ \mathcal{F}_{M,h_J}^{p}=\mathcal{F}_{M,h_J}^{1,p}$. 
For let $h_K\in L^p(K)$, the $f$-bounded extremal problem {\textsl{$f$-BEP}} consists in  finding $w_*\in \mathcal{F}_{M,h_J}^{f,p}$ such that:
\begin{align}\label{eq:f-BEP}
   \|h_K - w_*\|_{L^p(K)}=\text{min} \, \left\{\|h_K-w\|_{L^p(K)} \colon w \in \mathcal{F}_{M,h_J}^{f,p}\right\} \, .
\end{align}
Analogously to {\textsl{BEP}}, 
problem 	\eqref{eq:f-BEP} is still a particular case of the 
	general bounded extremal problems considered in \cite{Chalender2011}.

Let us denote by $T_J$ the operator defined on $L^p(J)$ by $T_J \,  h =T_{\D}(0 \vee h)$ for $h \in L^p(J)$, where $T_\D$ is defined by: \eqref{eq:Teodorescu_operator}. 
Notice that there exists a relation between the families of functions \eqref{eq:family_BEP} and \eqref{eq:family_f-BEP} provided by Lemma \ref{lemma:Pf}. More precisely, if $w\in \mathcal{F}_{M,h_J}^{f,p}$, then $w-T_{\D}[\alpha_f \overline{w}]\in \mathcal{F}_{M^*,h_J^*}^{p}$, with
        $h_J^*=h_J-T_J(\alpha_f \overline{h}_J)$
        and $M^*= M \, \varrho$, where $\varrho$ is equal to the norm of the operator $h \mapsto h-T_J(\alpha_f \overline{h})$  from $L^p(J)$ to itself. 
\begin{teorema}
The approximation set $\mathcal{F}_{M,h_J}^{f,p}|_K$ is closed in $L^p(K)$. Moreover, there exists a unique solution $w_*\in \mathcal{F}_{M,h_J}^{f,p}$ to the {\textsl{f-BEP}} \eqref{eq:f-BEP}.
\end{teorema}
\begin{proof}
Analogously to the proof of Lemma \ref{lemma-closed-1}, we obtain that $\mathcal{F}_{M,h_J}^{f,p}|_K$ is closed in $L^p(K)$. Thus the existence of the minimizing element comes from \cite[Part\ 3, Ch.\ II, Sec.\ 1, Prop.\ 5, \ 8]{Beauzamy1985}.  
Meanwhile, the uniqueness is proved as in Theorem \ref{th:existence-uniqueness-BEP}, using the generalized uniqueness principle for functions in $A_f^p(\D)$ (see Remark \ref{rmk:radial-limit-generalized}).
\end{proof}

By Proposition \ref{prop:dual-space} and the invariance formula \eqref{eq:invariance_formula_f}, we have that $P_f(\chi_J\cdot)$ and $P_f(\chi_K\cdot)$ are the adjoint operators of the restriction operators acting from $A_f^2(\D)$ to $A_f^2(J)$ and $A_f^2(K)$, respectively.
From the same reasons, $P_f(\chi_{\Omega}\cdot)$ is self-adjoint on $A^2_f(\D)$, for a (non-empty) domain $\Omega\subseteq \D$.
	We thus expect the solution $w_*$ to the {\textsl{$f$-BEP}} \eqref{eq:f-BEP} to be given in terms of the Bergman-Vekua projection $P_f$ in a similar manner than the solution $g_0$ to the {\textsl{BEP}} \eqref{eq:BEP} in the Bergman space  in terms of the Bergman projection $P$ in Theorem \ref{th:Theorem-solution-BEP}.

\section{Conclusion}
\label{sec:conclu}
In order to generalize the case $f\equiv 1$, boundedness and spectral properties of generalized Toeplitz operators  $P_f(\chi_\Omega \cdot)$  remain to be established 
in the Bergman-Vekua space $A_f^2(\D)$.

We possess little information about the Vekua-Bergman kernel of $A_f^p(\D)$, but when the conductivity has a separable form
(Example \ref{ex:sep}), some density properties could be expected, because the role of polynomials in 
$A^p(\D)$ is now played by the formal powers. 

Constructive aspects of \textsl{BEP} for $p \neq 2$ and of {\textsl{$f$-BEP}} will be considered, together with related numerical and resolution issues.

The present study could be extended to more general bounded simply connected Lipschitz domains, using conformal mappings \cite{Pommerenke1992}.

The results of Sections \ref{sec:Bergman-Vekua spaces} and \ref{sec:f-BEP} are still valid in Bergman spaces of solutions to the general Vekua equation \eqref{eq:Vekua-general}, with {$\alpha, \beta \in L^{\infty}(\D)$}.

\end{document}